\documentclass[11pt]{article}

\usepackage{articlestyle}
\usepackage{macro}
\usepackage{natbib} 
\usepackage{parskip, amsmath, amsfonts, amsthm, mathtools}

\newcommand{\Exp}[1]{\mathbb{E}\mathopen{}\left[{#1}\right]\mathclose{}}
\newcommand{\Expt}[1]{\mathbb{E}_t\mathopen{}\left[{#1}\right]\mathclose{}}
\newcommand{\ExpD}[2]{\mathbb{E}_{{#1}}\mathopen{}\left[{#2}\right]\mathclose{}}

\title{Last-Iterate Complexity of SGD for Convex and Smooth Stochastic Problems}

\author{%
  Guillaume Garrigos$^1$,
  Daniel Cortild$^2$, 
  Lucas Ketels$^{1,2}$, 
  Juan Peypouquet$^2$
  \vspace{0.5em} \\
  $^1$Université Paris Cité and Sorbonne Université, CNRS, \\ Laboratoire de
  Probabilités, Statistique et Modélisation, F-75013 Paris, France \\
  $^2$University of Groningen, Groningen, The Netherlands
  \vspace{0.5em} \\
  \texttt{ \{d.cortild, l.ketels, j.g.peypouquet\}@rug.nl, garrigos@lpsm.paris}
}

\renewcommand{\leq}{\le}

\begin{document}

\maketitle

\begin{abstract}
    Most results on Stochastic Gradient Descent (SGD) in the convex and smooth setting are presented under the form of bounds on the ergodic function value gap.
    It is an open question whether bounds can be derived directly on the last iterate of SGD in this context.
    Recent advances suggest that it should be possible. For instance, it can be achieved by making the additional, yet unverifiable, assumption that the variance of the stochastic gradients is uniformly bounded.
    In this paper, we show that there is no need of such an assumption, and that SGD enjoys a $\tilde O \left( T^{-1/2} \right)$ last-iterate complexity rate for convex smooth stochastic problems.
\end{abstract}

\section{Introduction}

We consider the stochastic optimization problem given by 
\begin{equation*}
    \min_{x\in \mathcal{H}} f(x),\quad \text{where }f(x)=\mathbb E_{i\sim \mathcal D}[f_i(x)],
\end{equation*}
where $\mathcal D$ is a distribution over the data. 
For instance, when $\mathcal D$ has finite support we recover the finite-sum minimization problem with $f=\frac{1}{m}\sum_{i=1}^mf_i(x)$.
In this paper, we focus on \emph{convex and smooth} stochastic problems, where each function $f_i\colon \mathcal{H} \to \mathbb{R}$ is convex and $L$-smooth, for some $L \in (0,+\infty)$.
Our objective is to provide complexity guarantees for the Stochastic Gradient Descent (SGD) algorithm \citep{robbins_stochastic_1951}, a widely used method for solving large-scale optimization problems.
The iterative update rule of SGD is defined as
\begin{equation}\label{Algo:SGD}\tag{SGD}
    x_{t+1}=x_t-\gamma \nabla f_{i_t}(x_t),
\end{equation}
where $i_t\sim \mathcal D$ is sampled i.i.d.~at each iteration,
and $\gamma>0$ is the step-size.

\paragraph{Ergodic complexity rates.} 
Standard analyses for SGD in the convex and smooth setting provide upper bounds on the expected ergodic function value gap $\Exp{f(\bar x_T) - \inf f}$, where $\bar x_T$ represents an average of the first $T$ iterates of the algorithm.
These upper bounds usually decompose into two components: a \emph{bias} term, which typically tends to zero, and a \emph{variance} term, which is often bounded and can be kept small by choosing a small enough step-size.
For example, as established in earlier works on SGD \citep{ nemirovski_robust_2009, bach_non-asymptotic_2011}, this algorithm enjoys bounds of the form
\begin{equation}\label{eq:rates generic}
    \Exp{f(\bar x_T) - \inf f} = O \left( \frac{1}{\gamma T} \right) + O \left( \gamma \sigma^2 \right).
\end{equation}
Those foundational studies relied on making an additional assumption on the stochastic gradients.
Such assumption could be a uniform bound on the variance of the gradients, given by
\begin{equation}\label{Ass:uniform variance eq}
    \sup_{x\in \mathcal{H}} \ \Exp{\|\nabla f_i(x) - \nabla f(x)\|^2} \leq \sigma^2,
\end{equation}
or a uniform bound on the (expected) square norm of the gradients, described by
\begin{equation}\label{Ass:uniform bounded gradients eq}
    \sup_{x\in \mathcal{H}} \ \Exp{\|\nabla f_i(x) \|^2} \leq \sigma^2.
\end{equation}
This framework has received continuous improvements, culminating in the recent contributions in \cite{taylor_stochastic_2019}, which provide sharp upper bounds, which are optimal in a certain sense.
Let us also mention that  significant efforts have been dedicated to extending bounds in expectation to high-probability guarantees. 
Namely, in \cite{liu2023highprobabilityconvergencestochastic}, the authors introduced a generic technique to establish high-probability convergence rates for the average optimality gap, under again quite restrictive variance related assumptions. 

Assumptions \eqref{Ass:uniform variance eq} and \eqref{Ass:uniform bounded gradients eq} were historically natural to make.
In its original formulation, the SGD algorithm was written as $x_{t+1} = x_t - \gamma (\nabla f(x_t) + \varepsilon_t)$, where $\varepsilon_t$ represented random noise.
Assuming finite variance for $\varepsilon_t$ was therefore a reasonable prerequisite.
However, the particular form of \eqref{Algo:SGD} implies that $\varepsilon_t$ is not any random vector, but precisely $\nabla f_{i_t}(x_t) - \nabla f(x_t)$.
While the bound of the gradient variance \eqref{Ass:uniform variance eq} holds true in the deterministic case, it is unclear how such property can be verified in practice for a true stochastic problem \citep{bottou_optimization_2018, nguyen_sgd_2018}.
The bound on the gradient norm \eqref{Ass:uniform bounded gradients eq} presents even greater difficulty. In the deterministic setting, this is equivalent to assuming the function $f$ to be Lipschitz continuous, and the class of convex Lipschitz and smooth functions is quite narrow.
Although, one does not need such bound to hold on the whole space but only at the generated iterates, this merely shifts the problem to verifying the boundedness of these iterates. However, this is equally hard to verify apriori, in particular when the step-size is constant.
This issue persists unless additional requirements, such as projections onto a compact domain, are enforced. 
Of course, assumptions \eqref{Ass:uniform variance eq} and \eqref{Ass:uniform bounded gradients eq} got relaxed with time.
These extensions typically aim to control the gradient or the variance with an upper bound depending on $x$, through linear combinations of $\Vert x \Vert^2$ and/or $\Vert \nabla f(x) \Vert^2$ \citep{blum_approximation_1954, gladyshev_stochastic_1965}, \citep[Assumption 2]{khaled_better_2023}, \citep[Assumption 4.3]{bottou_optimization_2018}.
We redirect the reader to  \cite{alacaoglu_towards_2025} and the references therein for a more exhausting description of those inequalities. 
However, the problem remains that these relaxations are impractical or impossible to verify for most problems.

\paragraph{Beyond variance assumptions.}

An interesting and recent line of research has been able to get rid of those assumptions.
This was initiated in \cite{bach_non-asymptotic_2011}, followed by the more recent works 
\cite{
needell_stochastic_2016, 
nguyen_sgd_2018, 
gower_sgd_2019,
khaled_better_2023,
gower_sgd_2021}.
The core to their analyses hinge on the convexity and smoothness of the functions $f_i$, or equivalently the cocoercivity of their gradients $\nabla f_i$.
This enables the derivation of a variance transfer inequality of the form
\begin{equation}\label{eq:variance transfer}
    \Exp{\Vert \nabla f_i(x) \Vert^2} \le O( \Exp{\Vert \nabla f_i(x_*) \Vert^2} + f(x) - \inf f), \quad \text{ for any } x_* \in \argmin f.
\end{equation}
This inequality is particularly useful as it allows us to bound all the variance terms by the constant $\sigma_*^2 \coloneqq \Exp{\Vert \nabla f_i(x_*) \Vert^2}$, at the price of obtaining an additional $f(x) - \inf f$ terms. This is generally not problematic as our objective is to derive bounds for this quantity.

This approach allowed to obtain bounds of the form \eqref{eq:rates generic} where $\sigma^2$ is replaced by $\sigma_*^2$, with no extra assumption than convexity and smoothness.
Those ideas extended to other variants of SGD, such as mini-batch SGD or non-uniform SGD \citep{gower_sgd_2019}.
The results progressively improved, up to \cite{cortild_new_2025} which provided bounds that are optimal in a certain sense, and allowing the step-sizes to cover the full range $\gamma L \in (0,2)$.
Inequality \eqref{eq:variance transfer} itself received some attention and got generalized into the ABC property \citep{khaled_better_2023}, which combines the features of \eqref{eq:variance transfer} and of previous variance assumptions.
Even when relaxing the convexity assumption, it was shown that this ABC property is enough for standard complexity results for SGD to remain true.

\paragraph{Last iterates.}

The above only discusses complexity rates for the \emph{ergodic} function value gap, and not the \emph{last iterates} function value gap $\Exp{f(x_T) - \inf f}$.
Obtaining guarantees for the last iterates is significant in practice, since it is the quantity the practitioner will consider.

The first result on the last iterate in the convex and smooth setting was established by \cite{bach_non-asymptotic_2011}, with improved bounds presented more recently in \cite[Theorem 5]{taylor_stochastic_2019} and \cite[Theorem 3.1]{liu_revisiting_2023}. While the former relies on a Lyapunov analysis guided by standard tools from the performance estimation problem framework, the latter builds on a proof from \cite{zamani_exact_2023}, which was originally developed to establish last-iterate convergence guarantees in deterministic non-smooth settings. 
Notably, all those works rely on  assumptions of uniformly bounded gradients or gradient variance.

It is also worth mentioning  existing results for convex {Lipschitz} stochastic problems.
In this framework, numerous results have established last-iterate convergence, both in expectation and with high-probability \citep{harvey_tight_2018,jain_making_2019}. 
Those studies required a bounded domain onto which the iterates of SGD are projected. However, it was shown in \cite{orabona_blog} that this bounded domain assumption could be lifted.
Note nevertheless that this framework is quite different from ours, and in particular that the Lipschitz assumption implies that the uniform bound on the gradients \eqref{Ass:uniform bounded gradients eq} holds.

Whether making a variance assumption is necessary remained an open question up to this day.
It was believed that an advantage of Momentum SGD over SGD is that the former naturally provides last-iterate results \citep{sebbouh2021almost, gower2025analysis}. Regularized SGD also enjoyed this advantage \citep{kassing2025controlling}. It seemed to be necessary to modify the algorithm to achieve such results, and that plain SGD cannot achieve last-iterate bounds without making a variance assumption.

In this paper, we answer this question, and show that SGD enjoys last-iterate guarantees without variance assumptions.  
Following a similar line as \cite{liu_revisiting_2023}, our work  adopts the techniques from \cite{zamani_exact_2023}, which we combine with a variance transfer inequality to remove the previously made variance assumption.
We show that for step-sizes $\gamma L \in (0,1)$, one can guarantee that 
\begin{equation*}
    \Exp{f(x_T) - \inf f} \le O \left( \frac{1}{\gamma T} + \gamma \ln(T) \sigma_*^2 \right)  T^{2 \gamma L}.
\end{equation*}
As a consequence, the classical choice $\gamma \simeq \tfrac{1}{\sqrt{T}}$ guarantees that
\begin{equation*}
    \Exp{f(x_T) - \inf f} \le O \left( \frac{\ln(T)}{\sqrt{T}} \right).
\end{equation*}
The rest of the paper is devoted to present this main result, its corollaries, and its proofs.

\paragraph{Note on a concurrent work.}
During the preparation of this manuscript, we became aware of the preprint \citep{attia_fast_2025}, which independently presents results very similar to those derived herein. 
In particular, \citep[Theorem 2]{attia_fast_2025} is analogous to our main result Theorem \ref{thm:main}, with the additional improvement that they allow for $\gamma L =1$.
We acknowledge that their work was made publicly available prior to ours.

\section{Problem setting and main assumptions}

Let $\mathcal H$ be a real Hilbert space with associated inner product $\langle \cdot, \cdot \rangle$ and induced norm $\|\cdot\|$. Let $\{f_i\}_{i\in \mathcal I}$ be a family of real-valued differentiable functions $f_i \colon \mathcal H \to \mathbb{R}$, where $\mathcal I$ is a (possibly infinite) set of indices. We consider the problem of minimizing $f := \Exp{f_i}$, where the expectation is taken over the indices $i \in \mathcal I$, with respect to some probability distribution $\mathcal{D}$ over $\mathcal I$. 
The following set of assumptions will be made throughout this paper:

\begin{assumption}[Convex and smooth problem]\label{ASS:1}\label{Ass:convex smooth problem}
With the notation introduced above, we impose the following:
\begin{enumerate}
        \item The problem is well-defined, in the sense that, for every $x \in \mathcal{H}$, the function $i \mapsto f_i(x)$ is $\mathcal{D}$-measurable, and $\Exp{f_i(x)}$ is finite.
        \item The problem is well-posed, in the sense that ${\rm{argmin}}~f \neq \emptyset$.
        \item The problem is convex, in the sense that each $f_i$ is convex.
        \item The problem is $L$-smooth for some $L \in (0,+\infty)$, in the sense that each gradient $\nabla f_i : \mathcal{H} \to \mathcal{H}$ is $L$-Lipschitz continuous. 
    \end{enumerate}
\end{assumption}

The key quantity which will appear in our bounds is the variance of the gradient $\nabla f_i$ at the minimizers.

\begin{assumption}[Solution Gradient Variance]\label{Ass:bounded solution variance}
    We assume that the variance at the solution exists, meaning that 
    \begin{equation}\label{D:bounded solution variance}
    \tag{GV$_*$}
        \mathbb{E}\left[ \Vert \nabla f_i(x_*) \Vert^2 \right] < + \infty
    \end{equation}
    for some $x_* \in {\rm{argmin}}~f$.    We will note
\begin{equation*}
    \sigma_*^2 \coloneqq  \inf\limits_{x_* \in {\rm{argmin}}~f} \Exp{\Vert \nabla f_i(x_*) \Vert^2}.
\end{equation*}
\end{assumption}

Even though $\sigma_*^2$ is defined above as an infimum, we recall from \citep[Lemma 4.17]{garrigos_handbook_2024} that under Assumption \ref{Ass:convex smooth problem}, we have $\sigma_*^2 = \mathbb{V}[\nabla f_i(x_*)]$ for every $x_* \in \argmin~f$.

The constant $\sigma_*^2$ is very important for our stochastic problem because it encodes partially how hard it is.
Indeed, $\sigma_*^2 \geq 0$ is a so-called interpolation constant \citep[Section 4.3]{garrigos_handbook_2024} which is zero if, and only if, interpolation holds, in the sense that all function $f_i$ share a common minimizer.
If interpolation holds, it is clear that our problem is easy, and that sampling one function or the other should not make much difference when running \eqref{Algo:SGD}.
Inversely, if $\sigma_*^2$ is large then the functions $f_i$ are likely to be very different from each other, meaning that the problem is harder, and this will be reflected in the complexity rates through this constant.

One could sense a contradiction between Assumption \ref{Ass:bounded solution variance} and our claim that we do not require a variance assumption.
But we stress here that this is not a variance assumption in the sense of controlling how the variance varies with $x$, as is done for instance in \eqref{Ass:uniform bounded gradients eq} or \eqref{Ass:uniform variance eq}.
Instead, here we only assume that the variance is defined at a single point.
Moreover, using Assumption \ref{Ass:convex smooth problem}, it can be verified in practice, under very mild assumptions, which cover most practical situations:
\begin{itemize}
    \item If $\mathcal{D}$ has finite support, then Assumption \ref{Ass:bounded solution variance} is trivially true.
    \item If the all the functions $f_i$ are nonnegative, then Assumption \ref{Ass:bounded solution variance} is true. This is more generally true if $\Exp{\inf f_i} > - \infty$, see Lemma \citep[Lemma A.10]{cortild_new_2025}.
    \item If the variance $\mathbb{V}\left[ \nabla f_i(x) \right]$ exists at any point $x \in \mathcal{H}$, then it exists at every point, see e.g. Lemma~\ref{L:variance everywhere or nowhere}. 
    In particular, Assumption \ref{Ass:bounded solution variance} is true.
\end{itemize}

\section{Main results}

We will now present our main last-iterate results for \eqref{Algo:SGD}.

\begin{theorem}[Generic step-size]\label{thm:main}
Let Assumptions \ref{Ass:convex smooth problem} and \ref{Ass:bounded solution variance} hold.
Let $T\ge 3$ be fixed, and let $(x_t)_{t=0}^T$ be generated by \eqref{Algo:SGD} with step size verifying 
$\gamma L \in (0, 1)$. Then
\[
    \Exp{f(x_T)-\inf f}\le T^{\phi}\left(\frac{2D^2}{\gamma (1-\gamma L) T} + \frac{8\gamma \ln(T+1)}{(1-\gamma L)^2}\cdot \sigma_*^2\right).
\]
where $D^2 = \Exp{\Vert x_0 - x_* \Vert^2}$  and $\phi=\frac{2\gamma L}{1+\gamma L}\in (0, 1)$.
\end{theorem}

We note that this yields the wanted bound of the order $O(T^{-1}+\ln(T))$, but with an additional multiplicative factor of $T^\phi$. 
But fortunately $\phi$ depends on the step-size itself, and it is quite easy to see that if the step-size has a mere dependency in $T$ then $T^\phi =  O(1)$ (see Lemma \ref{L:power of stepsize is not scary appendix} in the appendix).

\begin{lemma}[$T^\phi$ is not so scary]\label{L:power of stepsize not scary}
    Let $\phi=\frac{2\gamma L}{1+\gamma L}$ and $T \geq  2$.
    If $\gamma \leq \frac{K}{\ln(T)}$, then 
    $T^\phi \leq e^{2LK}$.
\end{lemma}

By taking a step-size of the order $\tfrac{1}{T^\beta}$, we obtain the following consequence of Theorem \ref{thm:main}, whose proof is given in Section \ref{coro:allalpha:proof}.

\begin{corollary}[Polynomial step-size]\label{coro:allalpha}
Let Assumptions \ref{Ass:convex smooth problem} and \ref{Ass:bounded solution variance} hold.
Let $T\ge 3$ be fixed, and let $(x_t)_{t=0}^T$ be generated by \eqref{Algo:SGD} with step-size $\gamma=\frac{1}{C L T^\beta}$, where $C \geq 2$ and $\beta \in (0, 1)$. Then 
    \[
        \Exp{f(x_T)-\inf f}\le  O\left(\frac{D^2}{T^{1-\beta}}+ \frac{\ln(T+1)}{T^\beta}\cdot\sigma_*^2\right),
    \]
    where $D^2 = \Exp{\Vert x_0 - x_* \Vert^2}$. The explicit constants hidden in the $O$ can be found in \eqref{E:corollaries}.
\end{corollary}

The bound in the above corollary is quite standard, and matches (up to the logarithmic factor) results which were previously obtained for ergodic bounds \citep{gower_sgd_2021}. 
As usual, the optimal choice for the exponent is given by $\beta=\tfrac12$. 
This statement follows in the same line as the previous corollary, and its proof may also be found in Section \ref{coro:allalpha:proof}. 

\begin{corollary}[Best polynomial step-size]\label{coro:alpha2}
    Let Assumptions \ref{Ass:convex smooth problem} and \ref{Ass:bounded solution variance} hold.
    Let $T\ge 3$ be fixed, and let $(x_t)_{t=0}^T$ be generated by \eqref{Algo:SGD} with step-size $\gamma=\frac{1}{C L\sqrt{T}}$, where $C \geq 2$.
    Then
    \begin{equation*} 
        \Exp{f(x_T)-\inf f}
        \le 
        \frac{9\cdot CLD^2}{\sqrt{T}} + \frac{67\cdot  \ln(T+1)}{CL\sqrt{T}}\cdot \sigma_*^2,   
    \end{equation*}
    where $D^2 = \Exp{\Vert x_0 - x_* \Vert^2}$. 
\end{corollary}

As a final direct consequence, we can provide complexity rates on the last iterates:

\begin{corollary}[Complexity rate] \label{coro:complexity}
    Let Assumptions \ref{Ass:convex smooth problem} and \ref{Ass:bounded solution variance} hold.
    Let $(x_t)_{t=0}^T$ be generated by \eqref{Algo:SGD}.
    For every $\varepsilon>0$, we can guarantee that 
    \begin{equation*}
        \Exp{f(x_T) - \inf f} \leq \varepsilon
    \end{equation*}
    provided that we take
    \begin{equation*}
        \gamma = \frac{1}{2L\sqrt{T}}, \text{ for some }, 
        \quad \text{ and } \quad 
        \frac{T}{(1 + \ln (T+1))^2 } \geq \frac{K^2}{\varepsilon^2},
    \end{equation*}
    where $K = \max\Big\{18 L \Exp{\|x_0 - x_*\|^2}, \tfrac{67 \sigma^2_*}{2L}\Big\}$.
    In particular, the above bound on $T$ is true if
    \[
    T \geq \frac{K'}{\varepsilon^{\beta}}, 
    \] 
    where 
    $K' = \left( \tfrac{3K}{e \alpha} \right)^\beta$
    and for any $\beta >2$.
\end{corollary}

Let us now provide some comments on those results.

\begin{remark}[About the tightness of the bound] \label{rem:tightness}
    As far as we know, it is not known whether the best possible last-iterate bound for \eqref{Algo:SGD} under Assumption \ref{Ass:convex smooth problem} is $O(1/\sqrt{T})$ or $O(\ln(T)/\sqrt{T})$.
    Therefore, it is worth looking at what has been done for convex and {Lipschitz} problems, which enjoys a rich literature with connections to online learning.
    For instance, \cite{harvey_tight_2018} show that if \eqref{Algo:SGD} is run with a vanishing step-size schedule $\gamma_t = {1}/{\sqrt{t}}$, then it is not possible to obtain a better bound than $\ln (T)/\sqrt{T}$.
    However, \cite{jain_making_2019} proved that with a non-standard choice of step-size the logarithmic dependency could be removed.
    We can then only conjecture that for convex smooth problems it is also possible to eliminate the $\ln(T)$ term.
\end{remark}

\begin{remark}[About (non-)adaptivity to smoothness]
    The results we obtained are not adaptive to the smoothness of the problem, in the sense that it is necessary to know the Lipschitz constant $L$ to set the step-size $\gamma$.
    Rules for defining the step-size which are adaptive to $L$ already exist for \eqref{Algo:SGD}, such as Adagrad \citep{streeter2010less} or Polyak step-sizes \citep{loizou2021stochastic, gower2025analysis, orabona2025new}.
    It would be interesting to know if those methods benefit from last-iterate guarantees.
\end{remark}

\begin{remark}[About (non-)adaptivity to interpolation]
    Our results cannot lead to bounds which are optimal \emph{and} adaptive with respect to interpolation.
    The presented bounds are trivially adaptive to $\sigma_*^2$ because we do not need to know it.
    But they are not optimal with respect to interpolation. 
    Indeed, an ideal bound would be
    \begin{equation} \label{e:interpolation ideal bound}
         O \left( \frac{D^2}{T} + \frac{\ln(T)}{\sqrt{T}} \cdot \sigma^2_*  \right).
    \end{equation}
    Such bound would mean that if $\sigma_*^2>0$ then the bound becomes $\tilde O(1/\sqrt{T})$ as usual.
    But if interpolation holds then the complexity switches to the $O(1/T)$ rate which is optimal for this scenario.
    As far as we know, there is no known result for SGD which is able to achieve \eqref{e:interpolation ideal bound} while at the same time being adaptive to interpolation.
    The only known way to obtain \eqref{e:interpolation ideal bound} is by knowing  (an estimate of) $\sigma_*^2$ and using this constant to define the step-size.
    With such knowledge, one could for instance set 
    \begin{equation*}
        \gamma = 
        \begin{cases}
            \frac{1}{4L \sqrt{T}} & \text{ if } \sigma_*^2 >0 \\
            \frac{1}{4L \ln(T)} & \text{ if } \sigma_*^2 = 0,
        \end{cases}
    \end{equation*}
    which will provide a $\tilde O \left( \tfrac{1}{T} + \tfrac{\sigma_*^2}{\sqrt{T}} \right)$ bound as a consequence of Theorem \ref{thm:main} and Corollary \ref{coro:alpha2}.
    Another standard choice (see e.g. \citep[Section D.2]{gower2025analysis}) could be
    \begin{equation*}
        \gamma = \frac{1}{4L\ln(T) \sqrt{1+ \sigma_*^2 T} }
    \end{equation*}
    which can also generate such a bound, using for instance \citep[Theorem D.1]{gower2025analysis} together with Theorem \ref{thm:main} and Lemma \ref{L:power of stepsize not scary}.
    As discussed in \citep[Section D.2]{gower2025analysis}, the constant $\sigma_*^2$ could be replaced with $2L \Delta_*$, if $\Delta_* \coloneqq \inf f - \mathbb{E}[ \inf f_i]$ itself can be computed.
    But that remains a challenge which could be as hard as minimizing $f$.
\end{remark}

\begin{remark}[Extensions to mini-batch SGD]
    All our results could be extended to mini-batch version of \eqref{Algo:SGD}.
    Indeed, as described in \citep[Section G]{gower_sgd_2019}, such mini-batch version can be seen as an instance of \eqref{Algo:SGD} itself, but applied to a different yet equivalent problem.
    The only consequence of this change would be that the constants defining the problem, namely $L$ and $\sigma_*^2$, will be updated through explicit formulas depending on the batch size. 
    
For instance, assume that support of $\mathcal{D}$ is finite and equal to $\mathcal I=\{1,\dots,n\}$, and pick a batch size $1 \leq b \leq n$. At each iteration, the mini-batch SGD algorithm computes
\begin{equation}\label{D:SGD miibatch}\tag{SGD$_b$}
    x_{t+1} = x_t - \frac{\gamma}{b} \sum_{i \in B_t} \nabla f_{i}(x_t),
\end{equation}
where $B_t$ is sampled independently, and uniformly among the subsets of $\mathcal I$ of size $b$. This algorithm is precisely \eqref{Algo:SGD} applied to  
\begin{equation*}
    \min_x \ f(x)  =\ExpD{\mathcal{B}}{\hat f_B(x)},
    \text{ where }
    \hat f_B(x) \coloneqq \frac{1}{b} \sum_{i \in B} \nabla f_{i}(x),
\end{equation*}
and where the expectation is taken with respect to the uniform law $\mathcal{B}$ over the set $\text{batch}_b$, which consists of all the subsets of $\mathcal I$ of size $b$.
If each $f_i$ is $L_i$-smooth, and $f=\sum_{i=1}^nf_i$ is $L_f$-smooth, the problem above satisfies Assumption \ref{ASS:1} with
$$L=\frac{n-b}{b(n-1)}L_f+\frac{n(b-1)}{b(n-1)}\max_iL_i$$
and
$$\sigma_*^2 =
\ExpD{\mathcal{B}}{\Vert \nabla f_B(x_*) \Vert^2}
= \frac{n-b}{nb(n-1)} \sum_{i=1}^n \Vert \nabla f_i(x_*) \Vert^2.$$
Further details can also be found in \citep[Appendix G]{cortild_new_2025}, which also contains the tools to extend such results to non-uniform sampling, such as importance sampling.
\end{remark}

\section{Proofs of the  main results}\label{thm:main:proof}

In our proofs, we will consider iterates $(x_t)_{t=0}^T$ generated by \eqref{Algo:SGD}.
We will denote $\mathcal F(x_0,\dots,x_t)$. the $\sigma$-algebra generated by $\{x_0,\dots,x_t\}$. 
We will also note $\mathbb E_t[Z]$ the conditional expectation of a random variable $Z$ with respect to $\mathcal F(x_0,\dots,x_t)$.

Our first main technical contribution is to obtain a bound of the form \eqref{e:li main energy} without uniform variance assumption. This will be the subject of Lemma \ref{lemma:tech1}. 
Once such a bound is obtained, we can obtain bounds on the last-iterate function gap. This will be presented in the subsequent Lemma \ref{lemma:technical2}. Using these two results we prove Theorem \ref{thm:main} in Section \ref{sec:mainproof}, and derive the remaining corollaries in Sections \ref{coro:allalpha:proof} and \ref{coro:allalpha:proof2}.
Before moving the  proof itself, let us state our main tool:

\begin{lemma}[Variance Transfer] \label{L:variance transfer}
Let Assumptions \ref{Ass:convex smooth problem} and \ref{Ass:bounded solution variance} hold true.
Let $x_*\in\argmin f$ and $x\in\mathcal H$. For every $\varepsilon>0$, we have
$$ \Exp{\|\nabla f_i(x)\|^2} \leq 2L(1+\varepsilon)(f(x) - \inf f) + \left(1 +\frac{1}{\varepsilon} \right)\Exp{\|\nabla f_i(x_*)\|^2}.$$
\end{lemma}

\begin{proof}
    This can be found for instance \citep[Lemma 4.20]{garrigos_handbook_2024}.
    Simply use a Fenchel-Young inequality
    \begin{equation*}
        \Exp{\|\nabla f_i(x)\|^2} \leq 
        (1+ \varepsilon) \Exp{\|\nabla f_i(x) - \nabla f_i(x_*)\|^2} + (1 + \varepsilon^{-1}) \Exp{\|\nabla f_i(x_*)\|^2} ,
    \end{equation*}
    and conclude after combining it with an expected smoothness inequality, which is a consequence of the convexity and smoothness of the functions $f_i$ (see e.g. \citep[Lemma 4.8]{garrigos_handbook_2024}):
    \begin{equation*}
        \frac{1}{2L}\Exp{\|\nabla f_i(x) - \nabla f_i(x_*)\|^2} 
        \leq 
        f(x) - \inf f.
    \end{equation*}
\end{proof}
 
\subsection{Lemma: Bounding a linear combination of function values}

\begin{lemma}\label{lemma:tech1}
    Let $f_i$ be convex and $L$-smooth, and let $(x_t)_{t=0}^T$ be generated by SGD with constant step-size $\gamma$ for $T\ge 1$. Then, for all $t=0, \ldots, T$ and $z_t\in \mathcal F(x_0, \ldots, x_t)$, it holds that 
    \begin{equation}\label{e:li main energy}
    \mathbb{E}\left[ a f(x_t) + b f(z_t) + c \inf f  \right]
    \leq 
    \frac{1}{2 \gamma} \mathbb{E} \Vert x_t - z_t \Vert^2 - \frac{1}{2 \gamma} \mathbb{E} \Vert x_{t+1} - z_t \Vert^2 + v, 
    \end{equation}
    where
    \[
        a = 1 - \gamma L (1+\varepsilon), \quad
        b = -1, \quad
        c = \gamma L (1+\varepsilon), \quad
        v = \frac{\gamma (1+\varepsilon^{-1}) \sigma_*^2}{2}
        \quad \hbox{and}\quad 
        \varepsilon=\frac{1-\gamma L}{1+\gamma L}.
    \]
\end{lemma}

\begin{proof}

Let $z_t\in \mathcal F(x_0, \ldots, x_t)$. For any $t \geq 0$ we write
\begin{equation*}
\Vert x_{t+1} - z_t \Vert^2 - \Vert x_{t} - z_t \Vert^2 
=
\Vert x_{t+1} - x_t \Vert^2 + 2 \langle x_{t+1} - x_t, x_t - z_t \rangle
=
\gamma^2 \Vert \nabla f_{i_t}(x_t) \Vert^2 + 2 \gamma \langle \nabla f_{i_t}(x_t), z_t - x_t \rangle.
\end{equation*}
Since each $f_i$ is convex, we can write
\begin{equation*}
    \Vert x_{t+1} - z_t \Vert^2 - \Vert x_{t} - z_t \Vert^2
    \leq 
    \gamma^2 \Vert \nabla f_{i_t}(x_t) \Vert^2
    + 2 \gamma \left( f_{i_t}(z_t) - f_{i_t}(x_t) \right).
\end{equation*}
Taking the expectation conditioned to $x_t$ 
and exploiting the fact that $z_t$ is independent from $x_t$ we obtain
\begin{equation}
    \label{e:li0}
    \mathbb{E}_t\Vert x_{t+1} - z_t \Vert^2 - \Vert x_{t} - z_t \Vert^2 
    \leq 
    \gamma^2 \mathbb{E}_t\Vert \nabla f_{i_t}(x_t) \Vert^2
    + 2 \gamma \left( f(z_t) - f(x_t) \right).
\end{equation}
The variance transfer Lemma \ref{L:variance transfer} states that
\begin{equation*}
\mathbb{E}\Vert \nabla f_{i_t}(x_t) \Vert^2 \leq 2(1+\varepsilon)L(f(x_t) - \inf f) + (1+\varepsilon^{-1}) \sigma_*^2,
\end{equation*}
for every $\varepsilon > 0$. After dividing by $2 \gamma$, our bound \eqref{e:li0} becomes
\begin{eqnarray*}
\frac{1}{2 \gamma}\mathbb{E}_t \Vert x_{t+1} - z_t \Vert^2 - \frac{1}{2 \gamma} \Vert x_{t} - z_t \Vert^2 
 \leq 
f(z_t) - f(x_t) 
+ \gamma (1+\varepsilon)L(f(x_t) - \inf f) 
+ \frac{\gamma (1+\varepsilon^{-1}) \sigma_*^2}{2}.
\end{eqnarray*}
Reorganizing terms, the above can be rewritten as
\begin{equation*} \label{E:abc_Exp}
    \Expt{a f(x_t) + b f(z_t) + c \inf f }
\leq 
\frac{1}{2 \gamma} \Expt{\Vert x_t - z_t \Vert^2} - \frac{1}{2 \gamma} \Expt{\Vert x_{t+1} - z_t \Vert^2} + v, 
\end{equation*}
where
\begin{equation*}\label{e:li abc constants short stepsizes}
    a = 1 - \gamma L (1+\varepsilon), \quad
    b = -1, \quad
    c = \gamma L (1+\varepsilon) \quad\hbox{and}\quad 
    v = \frac{\gamma (1+\varepsilon^{-1}) \sigma_*^2}{2}.
\end{equation*}
It is immediate from their definition that $a+b+c=0$.
On the other hand, we can make $a > 0$ if $\gamma L <1$, and $\varepsilon$ is taken small enough. A suitable choice is
\begin{equation*}
\varepsilon = \frac{1- \gamma L}{1+ \gamma L} 
\quad \Rightarrow \quad
a = 1 - \gamma L (1+ \varepsilon) = 1 - \frac{2 \gamma L}{1+\gamma L}= \frac{1- \gamma L}{1+ \gamma L}\quad \text{and}\quad v = \frac{\gamma \sigma^2_*}{1 - \gamma L}.
\end{equation*}
\end{proof}

\subsection{Lemma: From bounds on function values to last-iterate results}

The following result summarizes the technique used by \cite{zamani_exact_2023} and \cite{liu_revisiting_2023} in a slightly more general framework. The proof is heavily inspired by their original arguments.

\begin{lemma}\label{lemma:technical2}
    Let $f_i$ be convex and $L$-smooth, and let $(x_t)_{t=0}^T$ be generated by SGD with constant step-size $\gamma$ for $T\ge 1$.
    Suppose there exist $a,b,c\in \R$ with $-a< b\le 0$ and $a+b+c=0$ and $v\in \R_{\ge 0}$, such that it holds that for $t=0, \ldots, T$ and for all $z_t\in \mathcal H$ contained in $\mathcal F(x_0, \ldots, x_t)$,
    \begin{equation}\label{e:li1_new2}
    \mathbb{E}\left[ a f(x_t) + b f(z_t) + c \inf f  \right]
    \leq 
    \frac{1}{2 \gamma} \mathbb{E} \Vert x_t - z_t \Vert^2 - \frac{1}{2 \gamma} \mathbb{E} \Vert x_{t+1} - z_t \Vert^2 + v.
    \end{equation}
    Then it holds true that 
    \[
        \Exp{f(x_T)-\inf f}\le 
        \frac{\|x_0-x_*\|^2}{ 2\gamma a \alpha_{T-1}} + v \frac{\alpha_{T-1}+\sum_{t=0}^{T-1} \alpha_t}{a \alpha_{T-1}},
    \]
    where $(\alpha_t)$ is defined as $\alpha_{-1}=1$ and 
    \begin{equation*}
    \alpha_t = \frac{T-t+1}{T-t + 1+\tfrac{a}{b}} \cdot \alpha_{t-1} \quad \text{for $t=1, \ldots, T-1$}.
    \end{equation*}
\end{lemma}

\begin{proof}
We first wish to sum Inequality \eqref{e:li1_new2} over $t=0, \dots, T-1$ in a way that the right-hand side terms cancel each other.
To this end, assume first that $z_t \in [ x_t, z_{t-1}]$. Indeed, if $z_t = (1-p_t) x_t + p_t z_{t-1}$, with $p_t\in [0, 1]$, then 
\begin{equation*}
\Vert x_t - z_t \Vert^2 
=
\Vert p_t x_t - p_t z_{t-1} \Vert^2 = p_t^2 \Vert x_t - z_{t-1} \Vert^2
\leq 
p_t \Vert x_t - z_{t-1} \Vert^2.
\end{equation*}
After multiplication by $\alpha_t\ge 0$, Inequality \eqref{e:li1_new2} gives
\begin{equation*}
\alpha_t \mathbb{E}\left[ a f(x_t) + b f(z_t) + c \inf f  \right]
\leq 
\frac{1}{2 \gamma} \alpha_t p_t\mathbb{E}_t \Vert x_t - z_{t-1} \Vert^2 - \frac{1}{2 \gamma} \alpha_t \mathbb{E}_t \Vert x_{t+1} - z_t \Vert^2 + \alpha_t v.
\end{equation*}
Note that in order for this to hold for all $t\ge 0$, we must define $z_{-1}$, which we take to be $z_{-1}\coloneqq x_*$.
Assume that the sequence $(\alpha_t)$ is defined recursively starting from $\alpha_{-1}\coloneqq 1$, while verifying the relationship $\alpha_t p_t = \alpha_{t-1}$ for $t\ge 1$. Since $p_t \in [0,1]$, the sequence $\alpha_t$ is positive and nondecreasing.
Such choice of $\alpha_t$ leads to 
\begin{equation*}
\alpha_t \mathbb{E}\left[ a f(x_t) + b f(z_t) + c \inf f  \right]
\leq 
\frac{1}{2 \gamma} \alpha_{t-1}\mathbb{E}_t \Vert x_t - z_{t-1} \Vert^2 - \frac{1}{2 \gamma} \alpha_t \mathbb{E}_t \Vert x_{t+1} - z_t \Vert^2 + \alpha_t v,
\end{equation*}
which may now be summed
from $0$ to $T$ to obtain
\begin{equation*}
\sum_{t=0}^{T} \alpha_t \mathbb{E}\left[ a f(x_t) + b f(z_t) + c \inf f  \right]
\leq  
\frac{1}{2 \gamma} \alpha_{-1} \mathbb{E} \Vert x_0 - z_{-1} \Vert^2 - \frac{1}{2 \gamma} \alpha_{T} \mathbb{E} \Vert x_{T+1} - z_{T} \Vert^2 + v \sum_{t=0}^{T} \alpha_t.
\end{equation*}
Drop the negative term on the right-hand side, and recall that $\alpha_{-1}=1$ and $z_{-1}=x_*$. We are lead to
\begin{equation}\label{e:li2}
\sum_{t=0}^{T} \alpha_t \mathbb{E}\left[ a f(x_t) + b f(z_t) + c \inf f  \right]
\leq 
\frac{1}{2 \gamma}  
\mathbb{E} \Vert x_0 - x_* 
\Vert^2  + v \sum_{t=0}^{T} \alpha_t .
\end{equation}
We previously assumed that $z_t = (1-p_t) x_t + p_t z_{t-1}$ for some $p_t \in [0,1]$.
Unrolling this relationship yields
\begin{eqnarray*}
z_t &=& (1-p_t) x_t + p_t z_{t-1} \\
&=& (1-p_t) x_t + p_t (1-p_{t-1})x_{t-1} + p_t p_{t-1} z_{t-2} \\
& \vdots & \\
&=& \left( \sum_{s=0}^t p_t \dots p_{s+1}(1-p_s)  x_s \right) + (p_t p_{t-1} \dots p_0) z_{-1}.
\end{eqnarray*}
Now we will use the fact that $p_t = \tfrac{\alpha_{t-1}}{\alpha_t}$ to write
\begin{equation*}
z_t = \left( \sum_{s=0}^t \frac{\alpha_{s} - \alpha_{s-1}}{\alpha_{t}}   x_s \right) + \frac{\alpha_{-1}}{\alpha_t} z_{-1}=\left( \sum_{s=0}^t \frac{\alpha_{s} - \alpha_{s-1}}{\alpha_{t}}   x_s \right) + \frac{1}{\alpha_t} x_*.
\end{equation*}
Since $(\alpha_t)$ is nondecreasing and
\begin{equation*}
\frac{1}{\alpha_t} + \sum_{s=0}^t \frac{\alpha_{s} - \alpha_{s-1}}{\alpha_{t}}
=
\frac{1}{\alpha_t} \left( 1 + \sum_{s=0}^t (\alpha_{s} - \alpha_{s-1}) \right) = 1,
\end{equation*}
so that $z_t$ is a convex combination of $x_0, \ldots, x_t$ and $x_*$.

As such, we may upper bound $f(z_t)$ using Jensen's inequality as
\begin{equation*}
f(z_t) \le \left( \sum_{s=0}^t \frac{\alpha_{s} - \alpha_{s-1}}{\alpha_{t}}   f(x_s) \right)+\frac{1}{\alpha_t} f(x_*)=\left( \sum_{s=0}^t \frac{\alpha_{s} - \alpha_{s-1}}{\alpha_{t}}   f(x_s) \right)+\frac{1}{\alpha_t} \inf f.
\end{equation*}
We now recall that $b\le 0$, that $a+b+c=0$, and introduce the notation $r_t \coloneqq f(x_t) - \inf f$, so that
\begin{align*}
 \sum_{t=0}^{T} \alpha_t \left[ a f(x_t) + b f(z_t) + c \inf f  \right]
& \geq 
\sum_{t=0}^{T} \alpha_t \left[ a f(x_t) + b\left( \sum_{s=0}^t \frac{\alpha_{s} - \alpha_{s-1}}{\alpha_{t}}   f(x_s) \right) + b \frac{1}{\alpha_t} \inf f  + c \inf f  \right] \\
& = 
\sum_{t=0}^{T}  \left[ \alpha_ta f(x_t) + b\left( \sum_{s=0}^t ({\alpha_{s} - \alpha_{s-1}}) f(x_s) \right) + b \inf f + \alpha_t c \inf f  \right] \\
& = 
\sum_{t=0}^{T}  \left[ a \alpha_t r_t + b \left( \sum_{s=0}^t ({\alpha_{s} - \alpha_{s-1}}) r_s \right) \right] \\
& = 
\sum_{t=0}^{T} a \alpha_t r_t + b \sum_{t=0}^{T} ({\alpha_{t} - \alpha_{t-1}})  (T-t+1) r_t \\
& = 
\sum_{t=0}^{T} r_t \big( a \alpha_t  + b ({\alpha_{t} - \alpha_{t-1}}) (T-t+1) \big).
\end{align*}
In the last sum, we wish to make the coefficient in front of $r_T$ positive and all the remaining zero. Specifically, we assume $\alpha_{T-1}>\frac{a+b}{b}\alpha_T$ and, for all $t = 0, \dots, T-1$:
\begin{equation}\label{eq:relationship}
 a\alpha_t  = -b({\alpha_{t} - \alpha_{t-1}}) (T-t+1).
\end{equation}
Once we do this, and in view of Inequality \eqref{e:li2}, we will have proved that
\begin{equation*}
\left( (a+b) \alpha_T -b\alpha_{T-1} \right) r_T 
\leq
\frac{\|x_0-x_*\|^2}{2 \gamma} + v \sum_{t=0}^T \alpha_t.
\end{equation*}
Setting $\alpha_T=\alpha_{T-1}$, we get
\begin{equation*}\label{Proof::Pre}
r_T 
\leq
\frac{\|x_0-x_*\|^2}{ 2\gamma a \alpha_{T-1}} + v \frac{\alpha_{T-1}+\sum_{t=0}^T \alpha_t}{a \alpha_{T-1}}.
\end{equation*}
The relation for $\alpha_t$ in Equation \eqref{eq:relationship} can be rewritten as
\begin{equation*}
\alpha_t = \frac{T-t+1}{T-t + 1+\tfrac{a}{b}} \cdot \alpha_{t-1}.
\end{equation*}
Note that since $\tfrac{a}{b} \le 0$, $\alpha_t$ is increasing, which is consistent with the previous requirements.
\end{proof}

\subsection{Proof of Theorem \ref{thm:main}: Last iterates for generic step-size}\label{sec:mainproof}

Applying Lemmas \ref{lemma:tech1} and \ref{lemma:technical2}, we obtain 
\[
        \Exp{f(x_T)-\inf f}\le 
        \frac{\Exp{\|x_0-x_*\|^2}}{ 2\gamma a \alpha_{T-1}} + v \frac{\alpha_{T-1}+\sum_{t=0}^{T-1} \alpha_t}{a \alpha_{T-1}},
\]
where $\phi=1+\frac{a}{b}\in [0, 1]$ and $(\alpha_t)$ is defined as $\alpha_{-1}=1$ and 
\begin{equation*}
    \alpha_t = \frac{T-t+1}{T-t + 1+\tfrac{a}{b}} \cdot \alpha_{t-1}.
\end{equation*}
Combining Lemma \ref{lemma:gamma} and Lemma \ref{lemma:exponent}, we obtain 
\[
    \alpha_{T-1}\ge \frac{(T+1)^{1-\phi}}{2}\ge \frac{T^{1-\phi}}{2}\quad \text{and}\quad \frac{\alpha_T+\sum_{t=0}^{T-1} \alpha_t}{\alpha_{T-1}} \le 2\left(1+ \frac{T^\phi-1}{\phi}\right)+\frac{\alpha_T}{\alpha_{T-1}}\le 4 T^\phi\ln(T+1),
\]
such that
\[
    \Exp{f(x_T)-\inf f}\le \frac{\Exp{\|x_0-x_*\|^2}}{\gamma a T^{1-\phi}} + \frac{4vT^\phi\ln(T+1)}{a},
\]
where
\[
    \phi=1+\frac{a}{b}=\frac{2\gamma L}{1+\gamma L}\in (0, 1).
\]
Since
\[
    a=\frac{1-\gamma L}{1+\gamma L}\ge \frac{1-\gamma L}{2}\quad \text{and}\quad v = \frac{\gamma \sigma^2_*}{1 - \gamma L},
\]
we obtain the bound 
\[
    \Exp{f(x_T)-\inf f}\le \frac{2\Exp{\|x_0-x_*\|^2}}{\gamma (1-\gamma L) T^{1-\phi}} + \frac{8\gamma T^\phi\ln(T+1)}{(1-\gamma L)^2}\cdot \sigma_*^2.
\]

\subsection{Proof of Corollaries \ref{coro:allalpha} and \ref{coro:alpha2}: Last-iterate for polynomial step-size}\label{coro:allalpha:proof}

We adopt the notation from the proof of Theorem \ref{thm:main} and, for now, assume that $\gamma L\le 1/2$. Recalling that $\phi\le 2\gamma L$, the bound from Theorem \ref{thm:main} gives
\begin{equation} \label{E:corollaries_gamma}
    \Exp{f(x_T)-\inf f}\le \frac{4D^2}{\gamma T^{1-2\gamma L}} + 32\gamma T^{2\gamma L}\ln(T+1)\cdot \sigma_*^2.
\end{equation}

Suppose now that $\gamma = \frac{1}{CLT^\beta}$, with $C\ge 2$ (so that $\gamma L\le\frac{1}{2}$). Since $e\beta\ln(T)\le T^\beta$, we have
\[
    2\gamma L \ln T\le \frac{2}{e\beta C},
\]
whence
\begin{equation*} \label{E:Tphi_bounded}
   T^{2\gamma L}=\exp(2\gamma L \ln T)\le \exp\left(\frac{2}{e\beta C}\right)=:B. 
\end{equation*}
As a consequence,
$$\frac{1}{{\gamma} T^{1-{2\gamma L}}}=\frac{T^{2\gamma L}}{\gamma  T}\le \frac{BCL}{T^{1-\beta}}.$$
On the other hand, 
$$\gamma T^{2\gamma L} \le \frac{B}{CLT^\beta}.$$
Inequality \eqref{E:corollaries_gamma} then implies
\begin{equation} \label{E:corollaries}
    \Exp{f(x_T)-\inf f}\le \frac{4BCLD^2}{T^{1-\beta}} + \frac{32B\ln(T+1)}{CLT^\beta}\cdot \sigma_*^2,
\end{equation}
which proves Corollary \ref{coro:allalpha}. 
Finally, if $\beta=\tfrac12$, we conclude that
\begin{equation*} \label{E:Corollary_sqrtT}
    \Exp{f(x_T)-\inf f}\le \frac{4BCLD^2}{\sqrt{T}} + \frac{32 B \ln(T+1)}{CL\sqrt{T}}\cdot \sigma_*^2,   
\end{equation*}
where $B=\exp(\frac{4}{eC}) \leq \exp(\frac{2}{e}) < 2.09$, 
whence
\begin{equation*}
    \Exp{f(x_T)-\inf f}\le \frac{9CLD^2}{\sqrt{T}} + \frac{67 \ln(T+1)}{CL\sqrt{T}}\cdot \sigma_*^2,   
\end{equation*}

proving Corollary \ref{coro:alpha2}. For $C=2$, we can write
\begin{equation*} \label{E:Corollary_sqrtT_C2}
    \Exp{f(x_T)-\inf f}\le \frac{17LD^2}{\sqrt{T}} + \frac{34 \ln(T+1)}{L\sqrt{T}}\cdot \sigma_*^2.   
\end{equation*}

\subsection{Proof of Corollary \ref{coro:complexity}: Complexity bounds}\label{coro:allalpha:proof2}

For any $T \ge 1$, we have 
\begin{align*}
    \frac{1}{\varepsilon^2} &\leq \frac{T}{\max \Big\{18 L\Exp{\|x_0 - x_*\|^2}, \frac{67 \sigma^2_*}{2L}\Big\}^2 (1 + \ln (T+1))^2}
    \leq \frac{T}{\Big( 18 L\Exp{\|x_0 - x_*\|^2} + \frac{67 \sigma^2_*}{2L} \ln (T+1) \Big)^2},
\end{align*}
or, equivalently,
\begin{align*}
    \frac{18 L\Exp{\|x_0 - x_*\|^2} + \frac{67 \sigma^2_*}{2L} \ln (T+1)}{\sqrt{T}} \leq \varepsilon.
\end{align*}
From Corollary \ref{coro:alpha2}, it holds that $\Exp{f(x_T) - \inf f} \leq \varepsilon$, as wanted. Moving on to the second point, we are going to prove that $T \geq \tfrac{K'}{\varepsilon^{\beta}}$ is enough to reach an $\varepsilon$ precision, for some $K'$, where $\beta >2$.
Since $\beta >2$, there exists some $\alpha \in (0, 1/2)$ such that $\beta = 2/(1-2\alpha)$. 
If $T \geq \tfrac{K'}{\varepsilon^\beta}$, and defining $P = (K')^{1-2 \alpha}$, we have
\begin{equation*}
    \frac{P^{1/(1-2\alpha)}}{\varepsilon^{2/(1-2\alpha)}} \leq T 
    \iff  \frac{P}{\varepsilon^2} \leq T^{1-2\alpha} .
\end{equation*}
But we can write, using the fact that $T \geq 3$ :
\begin{equation*}
    1 + \ln(T+1) \leq 1 + 2 \ln(T) \leq 3 \ln(T) \leq \frac{3}{e \alpha} T^\alpha.
\end{equation*}
Therefore
\begin{equation*}
    T^{2 \alpha} \geq (1 + \ln(T+1))^2 \left( \frac{e \alpha}{3} \right)^2
    \quad \text{ and } \quad 
    \frac{1}{T^{2 \alpha}} \leq \frac{1}{(1 + \ln(T+1))^2 }\left( \frac{3}{e \alpha} \right)^2.
\end{equation*}
So we now have that
\begin{equation*}
    \frac{1}{\varepsilon^2} \leq \frac{T}{(1 + \ln(T+1))^2 }\left( \frac{3}{e \alpha} \right)^2 \frac{1}{P}.
\end{equation*}
Let us now assume that $P$ is such that 
\begin{equation*}
    \left( \frac{3}{e \alpha} \right)^2 \frac{1}{P} = \frac{1}{K^2}.
\end{equation*}
In that case
\begin{align*}
    \frac{1}{\varepsilon^2} \leq
    \frac{1}{K^2} \tfrac{T}{(1+\ln(T+1))^2}
    \leq 
    \frac{T}{\Big( 18 L \Exp{\|x_0 - x_*\|^2} + \frac{67 \sigma^2_*}{2L} \ln(T+1) \Big)^2},
\end{align*}
or, equivalently,
\begin{align*}
    \frac{18 L\Exp{\|x_0 - x_*\|^2} + \frac{67 \sigma^2_*}{2L} \ln(T+1)}{\sqrt{T}} \leq \varepsilon,
\end{align*}
and we conclude as previously.
So all we needed was to take
\begin{equation*}
    P = \left( \frac{3K}{e \alpha} \right)^2 
    \quad \text{ and } \quad 
    K' = \left( \frac{3K}{e \alpha} \right)^\beta.
\end{equation*}

\section{Conclusion}

In this paper, we provide the first last-iterate bounds for SGD without making a uniform variance assumption, and achieve a near-optimal complexity bound of $O(\tfrac{\ln T}{\sqrt T})$ with a step-size $\gamma\simeq \frac{1}{\sqrt T}$. 
We acknowledge the parallel work by \cite{attia_fast_2025}, who study the same problem and obtain similar results, and was made publicly available a few days prior to ours.

This new result creates opportunities for interesting possible extensions.
For instance, it is yet unknown if it is possible to obtain high-probability last-iterate bounds with no uniform gradient assumption, improving on the recent results in \cite{harvey_tight_2018,jain_making_2019, liu2023highprobabilityconvergencestochastic}. 
Moreover, it is not clear if the logarithmic dependency of our bounds is optimal.
More generally, a promising avenue could be to apply the performance estimation framework to characterize the worst-case bound, as was initiated in \cite{taylor_stochastic_2019} and \cite{cortild_new_2025}.


{
\small
\bibliographystyle{apalike}
\bibliography{references}

\begin{thebibliography}{}

\bibitem[Alacaoglu et~al., 2025]{alacaoglu_towards_2025}
Alacaoglu, A., Malitsky, Y., and Wright, S.~J. (2025).
\newblock Towards {Weaker} {Variance} {Assumptions} for {Stochastic}
  {Optimization}.
\newblock arXiv preprint arXiv:2504.09951.

\bibitem[Attia et~al., 2025]{attia_fast_2025}
Attia, A., Schliserman, M., Sherman, U., and Koren, T. (2025).
\newblock Fast {{Last-Iterate Convergence}} of {{SGD}} in the {{Smooth
  Interpolation Regime}}.
\newblock arXiv preprint arXiv:2507.11274.

\bibitem[Bach and Moulines, 2011]{bach_non-asymptotic_2011}
Bach, F. and Moulines, E. (2011).
\newblock Non-{Asymptotic} {Analysis} of {Stochastic} {Approximation}
  {Algorithms} for {Machine} {Learning}.
\newblock In {\em Advances in {Neural} {Information} {Processing} {Systems}},
  volume~24. Curran Associates, Inc.

\bibitem[Blum, 1954]{blum_approximation_1954}
Blum, J.~R. (1954).
\newblock Approximation {Methods} which {Converge} with {Probability} one.
\newblock {\em The Annals of Mathematical Statistics}, 25(2):382--386.

\bibitem[Bottou et~al., 2018]{bottou_optimization_2018}
Bottou, L., Curtis, F.~E., and Nocedal, J. (2018).
\newblock Optimization {Methods} for {Large}-{Scale} {Machine} {Learning}.
\newblock {\em SIAM Review}, 60(2):223--311.

\bibitem[Cortild et~al., 2025]{cortild_new_2025}
Cortild, D., Ketels, L., Peypouquet, J., and Garrigos, G. (2025).
\newblock New {Tight} {Bounds} for {SGD} without {Variance} {Assumption}: {A}
  {Computer}-{Aided} {Lyapunov} {Analysis}.
\newblock arXiv preprint arXiv:2505.17965.

\bibitem[Garrigos and Gower, 2024]{garrigos_handbook_2024}
Garrigos, G. and Gower, R.~M. (2024).
\newblock Handbook of {Convergence} {Theorems} for ({Stochastic}) {Gradient}
  {Methods}.
\newblock arXiv preprint arXiv:2301.11235.

\bibitem[Gautschi, 1959]{gautschi_elementary_1959}
Gautschi, W. (1959).
\newblock Some {{Elementary Inequalities Relating}} to the {{Gamma}} and
  {{Incomplete Gamma Function}}.
\newblock {\em Journal of Mathematics and Physics}, 38(1-4):77--81.

\bibitem[Gladyshev, 1965]{gladyshev_stochastic_1965}
Gladyshev, E.~G. (1965).
\newblock On {Stochastic} {Approximation}.
\newblock {\em Theory of Probability \& Its Applications}, 10(2):275--278.

\bibitem[Gower et~al., 2021]{gower_sgd_2021}
Gower, R., Sebbouh, O., and Loizou, N. (2021).
\newblock {SGD} for {Structured} {Nonconvex} {Functions}: {Learning} {Rates},
  {Minibatching} and {Interpolation}.
\newblock In {\em Proceedings of the 24th {International} {Conference} on
  {Artificial} {Intelligence} and {Statistics}}, pages 1315--1323. PMLR.

\bibitem[Gower et~al., 2025]{gower2025analysis}
Gower, R.~M., Garrigos, G., Loizou, N., Oikonomou, D., Mishchenko, K., and
  Schaipp, F. (2025).
\newblock {Analysis of an idealized stochastic Polyak method and its
  application to black-box model distillation}.
\newblock arXiv preprint arXiv:2504.01898.

\bibitem[Gower et~al., 2019]{gower_sgd_2019}
Gower, R.~M., Loizou, N., Qian, X., Sailanbayev, A., Shulgin, E., and
  Richtárik, P. (2019).
\newblock {SGD}: {General} {Analysis} and {Improved} {Rates}.
\newblock In {\em Proceedings of the 36th {International} {Conference} on
  {Machine} {Learning}}, pages 5200--5209. PMLR.

\bibitem[Harvey et~al., 2018]{harvey_tight_2018}
Harvey, N. J.~A., Liaw, C., Plan, Y., and Randhawa, S. (2018).
\newblock Tight {Analyses} for {Non}-{Smooth} {Stochastic} {Gradient}
  {Descent}.
\newblock arXiv preprint arXiv:1812.05217.

\bibitem[Jain et~al., 2019]{jain_making_2019}
Jain, P., Nagaraj, D., and Netrapalli, P. (2019).
\newblock Making the {Last} {Iterate} of {SGD} {Information} {Theoretically}
  {Optimal}.
\newblock arXiv preprint arXiv:1904.12443.

\bibitem[Kassing et~al., 2025]{kassing2025controlling}
Kassing, S., Weissmann, S., and D{\"o}ring, L. (2025).
\newblock {Controlling the Flow: Stability and Convergence for Stochastic
  Gradient Descent with Decaying Regularization}.
\newblock arXiv preprint arXiv:2505.11434.

\bibitem[Khaled and Richtárik, 2023]{khaled_better_2023}
Khaled, A. and Richtárik, P. (2023).
\newblock Better {Theory} for {SGD} in the {Nonconvex} {World}.
\newblock {\em Transactions on Machine Learning Research}.

\bibitem[Liu et~al., 2023]{liu2023highprobabilityconvergencestochastic}
Liu, Z., Nguyen, T.~D., Nguyen, T.~H., Ene, A., and Nguyen, H. (2023).
\newblock High {{Probability Convergence}} of {{Stochastic Gradient Methods}}.
\newblock In {\em Proceedings of the 40th {{International Conference}} on
  {{Machine Learning}}}. PMLR.

\bibitem[Liu and Zhou, 2023]{liu_revisiting_2023}
Liu, Z. and Zhou, Z. (2023).
\newblock Revisiting the {Last}-{Iterate} {Convergence} of {Stochastic}
  {Gradient} {Methods}.
\newblock In {\em Proceedings of {The} {Twelfth} {International} {Conference}
  on {Learning} {Representations}}.

\bibitem[Loizou et~al., 2021]{loizou2021stochastic}
Loizou, N., Vaswani, S., Laradji, I.~H., and Lacoste-Julien, S. (2021).
\newblock {Stochastic Polyak Step-size for SGD: An Adaptive Learning Rate for
  Fast Convergence}.
\newblock In {\em International Conference on Artificial Intelligence and
  Statistics}, pages 1306--1314. PMLR.

\bibitem[{Mathematics Stack Exchange}, 2017]{98348}
{Mathematics Stack Exchange} (2017).
\newblock {How do you prove Gautschi's inequality for the gamma function?}
\newblock \url{https://math.stackexchange.com/q/98348}.

\bibitem[Needell et~al., 2016]{needell_stochastic_2016}
Needell, D., Srebro, N., and Ward, R. (2016).
\newblock Stochastic gradient descent, weighted sampling, and the randomized
  {Kaczmarz} algorithm.
\newblock {\em Mathematical Programming}, 155(1):549--573.

\bibitem[Nemirovski et~al., 2009]{nemirovski_robust_2009}
Nemirovski, A., Juditsky, A., Lan, G., and Shapiro, A. (2009).
\newblock Robust {Stochastic} {Approximation} {Approach} to {Stochastic}
  {Programming}.
\newblock {\em SIAM Journal on Optimization}, 19(4):1574--1609.

\bibitem[Nguyen et~al., 2018]{nguyen_sgd_2018}
Nguyen, L., Nguyen, P.~H., Dijk, M., Richtarik, P., Scheinberg, K., and Takac,
  M. (2018).
\newblock {SGD} and {Hogwild}! {Convergence} {Without} the {Bounded}
  {Gradients} {Assumption}.
\newblock In {\em Proceedings of the 35th {International} {Conference} on
  {Machine} {Learning}}, pages 3750--3758. PMLR.

\bibitem[Orabona, 2020]{orabona_blog}
Orabona, F. (2020).
\newblock {Last Iterate of SGD Converges (Even in Unbounded Domains)}.
\newblock
  \url{https://parameterfree.com/2020/08/07/last-iterate-of-sgd-converges-even-in-unbounded-domains/}.

\bibitem[Orabona and D'Orazio, 2025]{orabona2025new}
Orabona, F. and D'Orazio, R. (2025).
\newblock {New Perspectives on the Polyak Stepsize: Surrogate Functions and
  Negative Results}.
\newblock arXiv preprint arXiv:2505.20219.

\bibitem[Robbins and Monro, 1951]{robbins_stochastic_1951}
Robbins, H. and Monro, S. (1951).
\newblock A {Stochastic} {Approximation} {Method}.
\newblock {\em The Annals of Mathematical Statistics}, 22(3):400--407.

\bibitem[Sebbouh et~al., 2021]{sebbouh2021almost}
Sebbouh, O., Gower, R.~M., and Defazio, A. (2021).
\newblock Almost sure convergence rates for stochastic gradient descent and
  stochastic heavy ball.
\newblock In {\em Conference on Learning Theory}, pages 3935--3971. PMLR.

\bibitem[Streeter and McMahan, 2010]{streeter2010less}
Streeter, M. and McMahan, H.~B. (2010).
\newblock Less regret via online conditioning.
\newblock arXiv preprint arXiv:1002.4862.

\bibitem[Taylor and Bach, 2019]{taylor_stochastic_2019}
Taylor, A. and Bach, F. (2019).
\newblock Stochastic first-order methods: non-asymptotic and computer-aided
  analyses via potential functions.
\newblock In {\em Proceedings of the 32nd {Conference} on {Learning} {Theory}},
  pages 2934--2992. PMLR.

\bibitem[Zamani and Glineur, 2023]{zamani_exact_2023}
Zamani, M. and Glineur, F. (2023).
\newblock Exact convergence rate of the last iterate in subgradient methods.
\newblock arXiv preprint arXiv:2307.11134.

\end{thebibliography}
}

\newpage

\appendix

\section{Appendix: Technical inequalities}

Most of our bounds involve complicated expressions that we want to simplify. Here are small tools that we need.

\begin{lemma}[Simplifying inequalities]\label{lemma:exponent}
    For every $t\ge 1$ and $\theta>0$, we have
    $$3+2\left(\frac{t^\theta-1}{\theta}\right)\le 4t^\theta\ln(t+1).$$
\end{lemma}

\begin{proof}
    Define $\psi(t)=3+2\left(\frac{t^\theta-1}{\theta}\right)- 4t^\theta\ln(t+1)$, so that $$\psi'(t)=2t^{\theta-1} -4\theta t^{\theta-1}\ln(t+1)-\frac{4t^\theta}{t+1} \le t^\theta\left(\frac{2}{t}-\frac{4}{t+1}\right)\le 0$$ 
    for every $t\ge 1$. Since $\psi(1)=3-4\ln(2)<0$, this implies $\psi(t)< 0$ for every $t\ge 1$, as claimed.
\end{proof}

\begin{lemma}\label{lemma:Exp}
    If $x\in [0, a]$, where $a> 0$, then it holds that
    \[
        \exp(x)\le x\cdot \frac{\exp(a)-1}{a}+1.
    \]
\end{lemma}
\begin{proof}
    From convexity of $\exp$ between points $0$ and $a$, we have for any $\alpha \in [0, 1]$
    \begin{align*}
        \exp(\alpha a) \leq \alpha \exp(a) + (1-\alpha)\exp(0)
    \end{align*}
    Set $x = \alpha a$, such that
    \begin{align*}
        \exp(x) \leq \frac{x}{a} \exp(a) + (1-\frac{x}{a}) \exp(0) = x \frac{\exp(a) - 1}{a} + 1.
    \end{align*}
    Since this is true for any $\alpha \in [0, 1]$, this is true for any $x \in [0, a]$.
\end{proof}

\begin{lemma}\label{lemma:gamma}
    For a fixed $T\ge 2$ and $\phi\in (0, 1]$, define $(\alpha_t)_{t=0}^{T-1}$ through $\alpha_{-1}=1$ and, for $t=0, \ldots, T$,
    $$\alpha_t=\frac{T-t+1}{T-t+\phi}\cdot \alpha_{t-1}.$$
    Then it holds that
    \[
        \alpha_{T-1}\ge \frac{(T+1)^{1-\phi}}{2}\quad \text{and}\quad \frac{\sum_{t=0}^{T-1} \alpha_t}{\alpha_{T-1}} \le 2\left(1+ \frac{T^\theta-1}{\theta}\right).
    \]
\end{lemma}
\begin{proof}
    Note that we may rewrite 
    $$\alpha_t=\frac{\Gamma(T+1+1)}{\Gamma(T+1+\phi)}\frac{\Gamma(T-t+\phi)}{\Gamma(T-t+1)},$$
    where $\Gamma(\cdot)$ represents the Gamma function. 
    By Gautschi’s Inequality\footnote{We initially found that bound thanks to \citep{98348}.} \citep{gautschi_elementary_1959},
     we have that
    \begin{equation*}
    (\forall x > 0)(\forall c \in [0,1]) \quad 
    x^{1-c} \le \frac{\Gamma(x+1)}{\Gamma(x+c)} \le (x+1)^{1-c}.
    \end{equation*}
    We can use this with $c = \phi\le 1$ and $x = T+1$ or $x = T-t$ to obtain
    \begin{equation*}
    \frac{(T+1)^{1-\phi}}{(T-t+1)^{1-\phi}} \le \alpha_t \le \frac{(T+2)^{1-\phi}}{(T-t)^{1-\phi}}.
    \end{equation*}
    Now we can proceed with the inequalities we need in our main bound.
    The simplest one is a lower bound for $\alpha_{T-1}$:
    \begin{equation*}
    \alpha_{T-1} \ge \frac{(T+1)^{1-\phi}}{2^{1-\phi}}.
    \end{equation*}
    The second bound we need is an upper bound on the sum of $\alpha_t$.
    This arrives from
    \begin{eqnarray*}
    \sum_{t=0}^{T-1} \alpha_t 
     \leq 
    \sum_{t=0}^{T-1}  \frac{(T+2)^{1-\phi}}{(T-t)^{1-\phi}} 
    =
    (T+2)^{1-\phi} \sum_{s=1}^{T}  s^{\phi-1}
    \leq 
    (T+2)^{1-\phi} \left( 1 + \int_1^T s^{\phi-1} ds \phi\right),
    \end{eqnarray*}
    where the last inequality is a sum-integrand bound, see for instance \cite{garrigos_handbook_2024}.
    \begin{equation*}
    \sum_{t=0}^{T-1} \alpha_t 
    \le (T+2)^{1-\phi}\left(1+\frac{T^\phi-1}{\phi}\right). 
    \end{equation*}
    Specifically, since $2\frac{T+2}{T+1}\le 3$ and $1-\phi\le 1$, the wanted bound follows.
\end{proof}

\begin{lemma}[$T^\phi$ is not so scary]\label{L:power of stepsize is not scary appendix}
    Let $\phi=\frac{2\gamma L}{1+\gamma L}$ and $T \geq 1$. For all $K\ge 0$, if $\gamma \leq \frac{K}{\ln T}$, then 
    $T^\phi \leq e^{2LK}$.
\end{lemma}
\begin{proof}
    From our assumptions, $\phi\le 2\gamma L \leq \tfrac{2L K}{\ln(T)}$, so $T^\phi =\exp(\phi \ln T)\le \exp(2L K)$.
\end{proof}

Finally, we prove a claim made earlier in the paper:

\begin{lemma}[Variance everywhere or nowhere]\label{L:variance everywhere or nowhere}
Let Assumptions \ref{Ass:convex smooth problem} and \ref{Ass:bounded solution variance} hold true. 
Then  
    \[
         \exists x \in \mathcal{H}, \ \Exp{\|\nabla f_i(x)\|^2} < +\infty \iff \forall x \in \mathcal{H}, \ \Exp{\|\nabla f_i(x)\|^2} \ .
    \]
\end{lemma}
\begin{proof}
    It suffices to prove that if there is $x\in \mathcal H$ such that $\Exp{\|\nabla f_i(x)\|^2}<\infty$, then $\Exp{\|\nabla f_i(y)\|^2}<\infty$ for every $y\in \mathcal H$.
    Indeed, suppose $\Exp{\|\nabla f_i(x)\|^2}<\infty$ for some $x\in\mathcal H$, and take $y\in \mathcal H$. Since
    $$\|\nabla f_i(y)\|^2\le 2\|\nabla f_i(x)\|^2+2\|\nabla f_i(y)-\nabla f_i(x)\|^2,$$
    the result is obtained by taking expectation and using an expected smoothness inequality, see e.g. \citep[Lemma 4.7]{garrigos_handbook_2024}.
\end{proof}

\end{document}